\documentclass[11pt,oneside]{amsart}
\usepackage{amsmath,amsthm,amsfonts,amssymb}
\usepackage[usenames]{color}
\usepackage{srcltx}
\usepackage{hyperref}
\usepackage{algorithm}
\usepackage{algorithmic}
\usepackage{graphicx}
\usepackage{enumerate}

\def\MZ{{\mathbb{Z}}}

\newtheorem{theorem}{Theorem}[section]
\newtheorem{lemma}[theorem]{Lemma}
\newtheorem{proposition}[theorem]{Proposition}
\newtheorem{corollary}[theorem]{Corollary}

\theoremstyle{definition}

\newtheorem{example}[theorem]{Example}

\newtheorem{remark}[theorem]{Remark}

\newcommand{\gp}[1]{\langle #1 \rangle}

\let\oldmarginpar\marginpar
\renewcommand\marginpar[1]{\-\oldmarginpar[\raggedleft\footnotesize #1]%
{\raggedright\footnotesize #1}}

\DeclareMathOperator{\Geodesic}{Geodesic}
\DeclareMathOperator{\Bouquet}{Bouquet}
\DeclareMathOperator{\Complete}{Complete}
\DeclareMathOperator{\Paths}{Paths}

\def\ovalpha{{\overline{\alpha}}}
\def\ovx{{\overline{x}}}

\newcommand{\coords}[1]{\ensuremath{\mathrm{Coord}({#1})}} 
\newcommand{\pivot}[1]{\ensuremath{\pi_{{#1}}}} 

\DeclareMathOperator{\Mal}{{Mal}}

\def\P{{\mathbf{P}}}

\newcommand{\am}{\noindent\color{blue} AM: }{}
{}

\title{Non-commutative lattice problems}

\author[]{Alexei Myasnikov, Andrey Nikolaev, and Alexander Ushakov}
\address{Stevens Institute of Technology, Hoboken, NJ, 07030 USA}
\email{amiasnik,anikolae,aushakov@stevens.edu}

\thanks{The work of the first and third author was partially supported by NSF grant DMS-1318716, the first author was also supported  by Russian Research Fund, project 14-11-00085.}

\begin{document}

\begin{abstract}
We consider several subgroup-related algorithmic questions in groups, modeled after the classic computational lattice problems, and study their computational complexity. We find polynomial time solutions to problems like finding a subgroup element closest to a given group element, or finding a shortest non-trivial element of a subgroup in the case of nilpotent groups, and a large class of surface groups and Coxeter groups. We also provide polynomial time algorithm to compute geodesics in given generators of a subgroup of a free group.


\end{abstract}

\maketitle

\section{Introduction}\label{sec:intro}

\subsection{Motivation}

In this paper, following \cite{MNU1,MNU2,Frenkel-Nikolaev-Ushakov:2014}  
we continue our research on  non-commutative discrete (combinatorial) optimization. Namely,  we define lattice problems  for an arbitrary algebraic structure and then study  these problems  together with their variations for an arbitrary group $G$.   The  purpose of this research is threefold.
First, we approach  lattice problems in a very different context by viewing them in the framework of classical algebra, thus facilitating  a deeper understanding of the nature of these  problems in general. Second, we try to unify and tackle  several interesting algorithmic problems in group theory that are related to lattice problems. Third, we aim to establish a unified outlook on several group-theoretic problems within the framework of lattice problems. We refer to \cite{MNU1}  for the initial motivation, the set-up of the problems, and initial facts on non-commutative discrete optimization. 

\subsection{Non-commutative lattice problems}

Let $G$ be a fixed finitely generated group with a fixed finite set of generators $A$. We  fix the word metric $d_A$ on $G$ relative to the generating set $A$ and for $g \in G$ by $|g|$ we  denote the length $d_A(g,1)$ of $g$ in the word metric $d_A$. A geodesic representing $g$ is a shortest word in $A\cup A^{-1}$ representing $g$ in $G$.

We always assume below, if not said otherwise, that elements of $G$ are given by words in the generators $A \cup A^{-1}$, and finitely generated subgroups of  $G$  are given by their finite generating sets, where elements are represented by words in $A\cup A^{-1}$.

All algorithmic problems we consider in this paper are related to word metrics,  geodesics, and distances between various subsets in $G$. The most fundamental one among them is the \emph{geodesic problem} in $G$ which requires for a given $g \in G$ to find a geodesic representing this $g$ in $G$. This problem is well-studied in geometric group theory and  there are many known results, but we mention only the following ones which clarify the nature of geodesics in groups that we consider  here: geodesics in hyperbolic groups~\cite{Epstein}, Coxeter groups~\cite{Bjorner, Brink-Howlett, Casselman, Diekert-K-L}, Artin groups~\cite{Holt-Rees}, and various solvable groups~\cite{Miasnikov_Romankov_Ushakov_Vershik:2010}.

\medskip\noindent
{\bf  Element-subgroup distance problem:} {\it  Given a  finitely generated subgroup $H$ of $G$ and an element $ g\in G$ find an element $h \in H$ which is closest to $g$ in the word metric on $G$. }

\medskip\noindent 
One can immediately recognize that this is precisely the non-commutative version of the classical \emph{closest element problem} in lattices.

\medskip\noindent 
{\bf Subgroup distance problem:}  {\it Given two finitely generated subgroups $H$ and $K$ of $G$  find the distance between them, i.e., find two elements  $h \in H$ and $k \in K$, not both trivial, with a shortest distance apart. }

\medskip\noindent 
Notice, that in the problem above it is very natural to  consider not only subgroups, but cosets, or double cosets, etc. The most general and powerful formulation here would be about distances between {\em rational subsets} of $G$.

\medskip\noindent 
{\bf Shortest  element problem in a subgroup:}  {\it Given a  finitely generated subgroup $H$ of $G$  find a shortest nontrivial element in $H$ in the word metric on $G$.}

\medskip\noindent 
Clearly, this is a general form  of the classical \emph{shortest vector problem}  in lattices.

\medskip\noindent 
{\bf Subgroup geodesic  problem:}
{\it Given a  subgroup $H$ of $G$ generated by $h_1, \ldots, h_n \in G$ and an element $ g\in G$ which belongs to $H$  find the geodesic length of $g$ with respect to the word metric on $H$ relative to the generating set $h_1, \ldots, h_n$.}

\medskip\noindent
Obviously, this is a very vast generalization of the initial geodesic problem in $G$ and in general it is much harder then the initial one. 



\subsection{Results and the structure of the paper }

In Section \ref{se:prelim} we prove some preliminary results on complexity of algorithmic problems in nilpotent groups and discuss some results due to Schupp on Coxeter and surface groups. 

In Section \ref{se:closest} we show that the closest element problem is decidable in polynomial time in  free groups, virtually nilpotent groups,  Coxeter and surface groups satisfying Reduction Hypothesis (all these groups are finitely generated). And in Section \ref{se:shortest} we show that for the same groups the shortest element problem is decidable in polynomial time as well.

In Section \ref{se:rational} we prove that for rational subsets of a free group, given by deterministic automata, the Rational Subsets Distance Problem  is decidable in polynomial time. In particular, this allows one to compute distances between subgroups or cosets in free groups.

In Section \ref{se:geodesic-subgroup} we solve the subgroup geodesic problem in free groups in polynomial time. Notice, that if the word problem in $G$ is decidable then the subgroup geodesic problem is also decidable (by  brute force verification), furthermore, if the word problem in $G$  is decidable in polynomial time and the subgroup is embedded isometrically then the brute force algorithm has an exponential running time. Thus the subgroup geodesic problem in surface groups, or limit groups, as well as in quasi-convex subgroups of hyperbolic groups, is decidable in exponential time. However, even in free groups solving this problem in polynomial time is a highly non-trivial  business.

\section{Preliminaries} \label{se:prelim}
\subsection{Nilpotent groups}
In this section we prove some basic facts regarding algorithmic complexity of certain problems in nilpotent groups. These facts appear to be well-known, but we were unable to find original explicit estimates of complexity. For the sake of completeness we provide them here.

Recall that in a group $G$ so-called {\em $N$-fold commutators} on a set $A \subseteq G$ are defined as follows. $1$-fold commutator on $A$ is any element of $A$. Inductively, $N$-fold commutator is any element $[u,v]$, where $u$ is an $i$-fold commutator, and $v$ a $j$-fold  commutator on $A$, where $i+j=N$.

Further, recall that a free class $c$ rank $r$ nilpotent group $N_{r,c}$ with basis $X = \{x_1,\ldots, x_r\}$ possesses a so-called Malcev basis, that is a tuple  of elements $Y = (y_1,\ldots,y_m)$ ($m$ is bounded above by a polynomial in $r$ of degree that linearly depends on $c$) such that \begin{enumerate}
\item every $y_i$ is a $k_i$-fold commutator on $X$, with $k_i\ge k_j$ whenever $i\ge j$; and
\item and every element $g\in N_{r,c}$ can be uniquely represented as
$$
g=y_1^{\alpha_1}\cdots y_m^{\alpha_m},\quad \alpha_i\in\mathbb Z.
$$
\end{enumerate}
The latter expression is called the {\em Malcev normal form} of $g$. The tuple
$[\alpha_1,\ldots,\alpha_m]$ is called the {\em Malcev exponent} of $g$. We write $[\alpha_1,\ldots,\alpha_m]=\Mal_Y(g)$ (or just $\Mal(g)$ when $Y$ is fixed) and  $g=Y^{[\alpha_1,\ldots,\alpha_m]}$. In what follows we assume for definiteness that $y_i=x_i$ for $i=1,\ldots, r$.

The following lemma is well-known (see~\cite[Theorem 6.5]{Hall:57}).
\begin{lemma}\label{le:malcev} Let $r,c$ be positive integers, and let $y_1,\ldots, y_m$ be a Malcev basis for the free nilpotent group $N_{r,c}$ of class $c$ and rank $r$.
\begin{itemize}
\item[(a)] Then there are polynomials $p_1,\ldots, p_m$ in variables $\alpha_1,\ldots,\alpha_m,\beta_1,\ldots,\beta_m$ such that
$$
(y_1^{\alpha_1}\cdots y_m^{\alpha_m})\cdot
(y_1^{\beta_1}\cdots y_m^{\beta_m})=
y_1^{p_1}\cdots y_m^{p_m},\quad \alpha_i,\beta_i\in\mathbb Z.
$$
\item[(b)] Further, if $0=\alpha_1=\ldots=\alpha_k$ or $0=\beta_1=\ldots=\beta_k$, then the corresponding values of $p_i$, $i=1,2\ldots,k+1$, are $\alpha_i+\beta_i$.
\item[(c)] There are polynomials $q_1,\ldots, q_m$ in variables $n,\alpha_1,\ldots, \alpha_m$ such that
$$
(y_1^{\alpha_1}\cdots y_m^{\alpha_m})^n=y_1^{q_1}\cdots y_m^{q_m}, \quad n,\alpha_i\in\mathbb Z.
$$
\end{itemize}
\end{lemma} 

\begin{lemma}\label{le:compute_malcev}Let $r,c$ be positive integers, and let $y_1,\ldots, y_m$ be a Malcev basis for the free nilpotent group $N_{r,c}$ of class $c$  and rank $r$. There is an algorithm that given a word $g$ in free generators of $N_{r,c}$ computes the Malcev normal form of $g$ in a time polynomial in the word length $|g|$ of $g$.
\end{lemma}
\begin{proof}
We prove the following statement by induction in $k$: if $w$ is a group word in variables $y_k,y_{k+1},\ldots, y_m$, $1\le k\le m$, then its Malcev normal form $y_k^{\alpha_k}\cdots y_m^{\alpha_m}$ can be computed as a group word in a time polynomial in the word length of $w$. 

Base of induction $k=m$ is obvious. Suppose the statement holds for $k=j+1$, and prove it for $k=j$. Given a word $w$ in variables $y_j,\ldots, y_m$, represent it as
$$
w=w_1 y_j^{\beta_1}w_2y_j^{\beta_2}\cdots w_\ell y_j^{\beta_\ell}w_{\ell+1},
$$
where $w_1,\ldots, w_{\ell+1}$ are (possibly trivial) group words in variables $y_{j+1},\ldots, y_m$ and $\beta_1,\ldots,\beta_\ell$ are nonzero integers. Note that both $\ell$ and $|w_i|$ are bounded by $|w|$, so after a polynomial time computation we can assume that all $w_i$ are given by their Malcev normal forms. Then we ``push'' all occurrences of $y_j$ to the left, starting with the rightmost occurrence, as follows.
\begin{itemize}
\item At the first step, we apply Lemma~\ref{le:malcev}(a) to compute Malcev normal form of $w_\ell\cdot y_j^{\beta_\ell}$. Note that by Lemma~\ref{le:malcev}(b), this Malcev normal form is a word $y_j^{\beta_\ell}w'_\ell$, where $w'_\ell$ is a group word in variables $y_{j+1},\ldots, y_m$. This allows us to represent
$w$ as
\[
w=w_1 y_j^{\beta_1}\cdots w_{\ell-1}y_j^{\beta_{\ell-1}+\beta_\ell}w'_\ell w_{\ell+1}.
\]
Note that the length of $w'_\ell$ is polynomial in $|w|$ by Lemma~\ref{le:malcev}(a).
\item Subsequently, having obtained
\[
w=w_1 y_j^{\beta_1}\cdots w_{\ell-i}y_j^{\beta_{\ell-i}+\ldots+\beta_\ell}w'_{\ell-i+1}\cdots w'_\ell w_{\ell+1},\quad 1\le i\le \ell-1,
\]
we apply Lemma~\ref{le:malcev}(a,b) to compute Malcev normal form of $w_{\ell-i}\cdot y_j^{\beta_{\ell-i}+\ldots+\beta_\ell}$, obtaining
\[
w=w_1 y_j^{\beta_1}\cdots w_{\ell-i-1}y_j^{\beta_{\ell-i-1}+\ldots+\beta_\ell}w'_{\ell-i}w'_{\ell-i+1}\cdots w'_\ell w_{\ell+1},
\]
Note that the length of $w'_{\ell-i}$ is polynomial in $|w|$ by Lemma~\ref{le:malcev}(a).
\item Repeating the above step $\ell-1\le |w|$ times, we arrive at
\[
w=y_j^{\beta_1+\beta_{2}+\ldots+\beta_\ell}w'_1 w'_2\cdots w'_\ell w_{\ell+1}.
\]
Now it is only left to observe that the word $w'_1 w'_2\cdots w'_\ell w_{\ell+1}$ in variables $y_{j+1},\ldots, y_m$ is of polynomial length in $|w|$, so by the induction assumption its Malcev normal form can be computed in polynomial time.
\end{itemize}
Note that the degree of polynomial that bounds time complexity of the above procedure (possibly) grows as $j$ decreases, but since $m$ only depends on $r,c$, the degree of polynomial that bounds time complexity of the resulting algorithm ultimately  depends only on $r,c$.
\end{proof}
\begin{remark}
\ 
\begin{enumerate}
\item The above procedure can be significantly optimized by taking into account structure of the Malcev basis $Y$ (see~\cite{MMNV}).
\item It immediately follows that the exponents in $\Mal(g)$ are bounded by a polynomial in $|g|$ (that depends on $r,c$).
\item Lemmas~\ref{le:malcev}, \ref{le:compute_malcev}, \ref{le:magic} with minimal changes are also true for appropriate bases of arbitrary finitely generated  nilpotent groups by essentially the same argument (see~\cite{MMNV}).
\end{enumerate}
\end{remark}

\begin{lemma}\label{le:magic} Let $r,c$ be positive integers, and let $y_1,\ldots, y_m$ be a Malcev basis for the free nilpotent group $N_{r,c}$ of class $c$  and rank $r$. There is an algorithm that, given elements $g_1,\ldots,g_n\in N_{r,c}$ and exponents $k_1,\ldots,k_n\in\mathbb Z$, computes the Malcev normal form of the element $g_1^{k_1}\cdots g_n^{k_n}$ in a time polynomial in $N=n+\sum |g_j|+\sum |k_j|$.
\begin{proof}Note that the word length of an element $g_j^{k_j}$ is bounded by $|g_j||k_j|\le N^2$. Further, since $n\le N$, the word length of $g_1^{k_1}\cdots g_n^{k_n}$ is bounded by $N^2\cdot N=N^3$. By Lemma~\ref{le:compute_malcev}, the Malcev normal form of such element can be computed in a time polynomial in $N^3$.
\end{proof}

\end{lemma}
The ``noncommutative Gauss'' algorithm for solving membership problem in nilpotent (or, more generally, polycyclic) groups is well-known~\cite{LNS}. In the following lemma we investigate the complexity of this algorithm.
\begin{lemma}\label{le:nilpotent_gauss}
Subgroup membership problem in a finitely generated nilpotent group is decidable in polynomial time.
\end{lemma}
\begin{proof}
Since every subgroup of a finitely generated free nilpotent group is finitely generated, it is enough to prove the statement in the case of a free nilpotent group.

Let $N_{r,c}$ be the free nilpotent group of rank $r$ and class $c$. Since the Malcev normal form can be computed in polynomial time by Lemma~\ref{le:compute_malcev}, we assume that a subgroup elements $h,h_1,\ldots,h_n$ are given by their Malcev normal forms.  For a tuple $h_1,\ldots, h_n$, we form the coordinate matrix $A$, that is an $n\times m$ matrix whose $i$-th row is the Malcev exponent of $h_i$:
\[
A=\left( 
\begin{array}{ccccc}
\alpha_{11} & \alpha_{12} &\ldots & \alpha_{1m}\\
\vdots &&&\vdots\\
\alpha_{n1} & \alpha_{n2} &\ldots & \alpha_{nm} 
\end{array}
\right),
\]
where $h_i=Y^{[\alpha_{i1} , \alpha_{i2},\ldots,\alpha_{im}]}$, $i=1,\ldots,n.$ 
We say that the matrix $A$ is in {\em triangular} form if
it has the following properties ($\pivot i$ denotes so-called pivot, i.e., the position of the first nonzero element in row $i$):
\begin{enumerate}[(i)]
\item All rows of $A$ are non-zero (i.e. no $h_{i}$ is trivial).\label{li:std_nontrivial}
\item $\pivot{1} < \pivot{2} < \ldots < \pivot{s}$ (where $s$ is the number of pivots).\label{li:std_echelon}
\end{enumerate}
The tuple $h_1,\ldots,h_n$ is called {\em full} if the corresponding matrix is triangular and in addition
\begin{enumerate}[(i)]\setcounter{enumi}{2}
\item \label{li:std_full}  
$H\cap \langle a_i,a_{i+1},\ldots, a_m\rangle$ is generated by $\{h_{j}\mid \pi_j\ge i\}$, 
for all $1\le i\le m$.
\end{enumerate}

In (\ref{li:std_full}), note that $\{h_{j}\mid \pi_j\ge i\}$ consists of the elements that have 0 in their first $i-1$ coordinates. 
It follows from Lemma~\ref{le:malcev}(b) that (\ref{li:std_full}) 
holds for a given $i$ if and only if the following property holds.
\begin{itemize}
\item[(iii)'] For all $1\leq k<j\leq s$ with $\pivot{k}<i$, 
$h_{j}^{h_{k}}$ and $h_{j}^{h_{k}^{-1}}$ are 
elements of $\langle h_{l}\,|\, l>k\rangle$. 
\end{itemize}

To solve membership problem for a given input $(h,h_1,\ldots,h_n)$ we start by forming the coordinate matrix $A_{0}$ for the tuple $h_1,\ldots,h_n$.  We produce matrices $A_{1}, \ldots, A_{s}$, with $s$ the number of pivots in the triangular full form of $A_{0}$, such that for every $k=1,\ldots, s$ the first 
$\pi_k$ columns of $A_{k}$ form a matrix satisfying (\ref{li:std_echelon}), and the condition (\ref{li:std_full}) is satisfied for all $i<\pi_{k+1}$, so that 
$A_{s}$, upon discarding trivial rows, is the triangular full form of $A_{0}$. Here we formally denote $\pivot{s+1}=m+1$.

Let $A_{k-1}$, $k\ge 1$, be constructed. Below we construct the matrix $A_k$, starting by setting $A_k=A_{k-1}$.  Below we let $h_1,h_2,\ldots$ denote the group elements represented by the corresponding rows of the matrix $A_k$, and $\alpha_{ij}$ the entry $(i,j)$ of $A_k$.

First, we identify the column of the next pivot $\pi=\pi_k$, that is the first column with at least one nonzero entry in rows $i\ge k$. Compute a linear expression of $d=\gcd(\alpha_{k\pi},\ldots,\alpha_{n\pi})$:
\[
d=l_k\alpha_{k\pi}+\cdots +l_n\alpha_{n\pi}.
\]
Using Lemma~\ref{le:magic}, we compute the Malcev exponent of the group element $h_{n+1}=h_k^{l_k}\cdots h_n^{l_n}$,
\[
\Mal(h_{n+1})=[0,\ldots, 0,\ d=\alpha_{n+1, \pi},\ \alpha_{n+1, \pi+1},\ \ldots,\ \alpha_{n+1, m}].
\]
Then we
\begin{enumerate}
\item add the above row to the matrix $A_k$;
\item for each $i=k,\ldots, n$, replace $i$-th row by $\Mal\left(h_i\cdot h_{n+1}^{-\alpha_{i1}/d}\right)$;
\item rearrange rows $k,\ldots,n+1$ of the obtained matrix so that the only nonzero entry in the first column in those rows is in the row $k$. 
\end{enumerate}
Note that operations (1)--(3) above preserve preserve the group generated by rows of the matrix, and by Lemma~\ref{le:malcev}(a,c) can be done in polynomial time in terms of the entries of $A_{k-1}$; and that the property~(\ref{li:std_echelon}) with $i\le \pivot k$ holds for $A_1$ by Lemma~\ref{le:malcev}(b).

To obtain~(\ref{li:std_full}) for $i<\pivot{k+1}$, we identify the next pivot $\pivot{k+1}$, setting $\pivot{k+1}=m+1$ if 
$\pivot{k}$ is the last pivot. 
We now ensure condition (\ref{li:std_full}) for $i< \pi_{k+1}$. 
Observe that operations (1)--(3) above preserve $\langle h_{j}\,|\,\pivot{j}\geq i\rangle$ for all 
$i<\pivot{k}$. Hence (\ref{li:std_full}) holds in $A_{k}$ for $i<\pivot{k}$ since it holds in $A_{k-1}$ for the same range.  Now consider $i$ in the range 
$\pivot{k}\le i<\pivot{k+1}$. It suffices to provide~(\ref{li:std_full})' for all $j>k$.


To obtain~(\ref{li:std_full})', we notice that $h_{k}^{-1}h_{j}h_{k},h_{k}h_{j}h_{k}^{-1}\in \langle h_\ell \mid \ell>k\rangle$ if and only if 
$[h_{j},h_{k}^{\pm 1}]\in \langle h_{\ell}\mid \ell>k\rangle$. 
Further, note that the subgroup generated by the set 
\[
S_{j}=\{1,h_j, [h_j,h_k],\ldots, [h_j,h_k,\ldots,h_k]\}, 
\]
where $h_k$ appears $m-\pivot{k}$ times in the last commutator, is closed under commutation with $h_k$ since if $h_{k}$ appears more than $m-\pivot{k}$ times then 
the commutator is trivial. An inductive argument shows that the subgroup $\langle S_{j}\rangle$ coincides with $\langle h_j^{h_k^\ell}\mid 0\le \ell\le m-\pi_k\rangle$. Similar observations can be made for conjugation by $h_k^{-1}$. Therefore, appending rows $\coords{h_j^{h_k^\ell}}$ for all 
$1\leq |\ell| \leq m-\pivot{k}$ and all $k< j\le n+3$ delivers~(\ref{li:std_full})' for all $j>k$. Note that~(\ref{li:std_full})' remains true for $i<\pivot{k}$. 


The process terminates at a matrix $A_s$, $s\le m$. Discarding trivial rows, we obtain a matrix $A$ that satisfies~(\ref{li:std_nontrivial})--(\ref{li:std_full}), whose rows generate the same subgroup as those of $A_0$. Finally, observe that for a triangular full matrix $A$, checking whether $h\in H$ can be done straightforwardly. Indeed, let $\Mal(h)=(\alpha_1,\ldots,\alpha_m)$ and let $\alpha_\pi$ be the first nonzero coordinate of $h$. Then, by Lemma~\ref{le:malcev} and property~(\ref{li:std_full}) of $A$, $h\in H$ if and only if $h^{(1)}=hh_1^{\alpha_\pi/\alpha_{1\pi}}\in H$ (in particular, if $\alpha_\pi$ is not a multiple of $\alpha_{1\pi}$, then $h\notin H$). Proceed ``left to right'', successively eliminating entries of $h^{(i)}$. If at any step the elimination fails, $h\notin H$. If the process terminates at a trivial $h^{(s)}$, then $h\in H$, and moreover, a representation $h=h_1^{l_1}\cdots h_s^{l_s}$ has been found.
\end{proof}
\begin{remark}
Observe that the above algorithm is presentation-uniform, that is, for a fixed $r$ and $c$, given a presentation with at most $r$ generators of a class at most $c$ nilpotent group, and input of the membership problem, it in polynomial time decides membership in the group given by the presentation for the given elelements.
\end{remark}

\subsection{Surface groups and Coxeter groups}
In~\cite{Schupp:2003}, Schupp gives an analogue of Stallings graphs for free groups in the case of Coxeter and surface groups satisfying certain conditions. Remarkably, the corresponding algorithms have polynomial time complexity. To formulate the respective statements, we need to provide certain small cancellation conditions.

Recall that a {\em Coxeter group} $G$ is a group with presentation
\begin{equation}\label{eq:cox}
G=\langle A\mid R\rangle=\langle a_1,\ldots, a_n\mid a_i^2,\ (a_ia_j)^{m_{ij}}, i\neq j\rangle,
\end{equation}
where $m_{ij}=m_{ji}>1$ and we may have $m_{ij}=\infty$ (which denotes the absence of a defining relator between $a_i$ and $a_j$). The above presentation is referred to as the standard presentation of the Coxeter group $G$. For each $1\le i\le n$, let $\rho_i$ denote the number of indices $j\neq i$ s.t. $m_{ij}<\infty$. Also set $\rho_{ij}=\max\{\rho_i,\rho_j\}$.

We say that a Coxeter group $G$ given by its standard presentation satisfies the {\em Reduction Hypothesis} if  in~\eqref{eq:cox}, $n\ge 3$ and each $m_{ij}>4$, and there is a subset $C\subseteq A$ such that every defining relator $(a_ia_j)^{m_{ij}}$ contains a generator from $C$ and satisfies the following condition:
\begin{enumerate}
\item If both $a_i,a_j\in C$, then $m_{ij}>\tfrac{3}{2}\rho_{ij}$;
\item If $a_i\in C$ and $a_j\notin C$ then $m_{ij}>2\rho_i$.
\end{enumerate}

We say that a {\em surface group} $G$ given by its standard presentation satisfies the {\em Reduction Hypothesis} if the standard defining relator has length at least $8$, i.e. the genus is at least $2$ in the orientable case and is at least $4$ in the non-orientable case.

\begin{theorem}[Schupp,~\cite{Schupp:2003}]\label{th:coxeter_graph}
There is a fixed quadratic time algorithm which, when given the standard presentation of a Coxeter group $G=\langle A\mid R\rangle$ satisfying the Reduction Hypothesis, a tuple $(h_1,\ldots,h_m)$ of generators for a subgroup $H$, 
calculates the graph $\Delta(H)$ which, in particular, is the graph of a finite state automaton which accepts a Dehn reduced word $u$ if and only if $u\in H$.
\end{theorem}

\begin{theorem}[Schupp,~\cite{Schupp:2003}]\label{th:surface_graph}
There is a fixed quartic time algorithm which, when given the standard presentation of a surface group of genus at least $2$ in the orientable case and at least $4$ in the non-orientable case, and a tuple $(h_1,\ldots,h_m)$ of generators for a subgroup $H$, 
calculates the graph $\Delta(H)$ which, in particular, is the graph of a finite state automaton which accepts a Dehn reduced word $u$ if and only if $u\in H$.
\end{theorem}

\section{Closest Element Problem}
\label{se:closest}




\begin{theorem}\label{th:free_closest}
The closest element problem in free groups of finite rank is solvable in polynomial time.
\end{theorem}
\begin{proof}
Let $h_1,h_2,\ldots,h_n\in F$ be words in free generators of a free group $F$. One can construct in polynomial time the Stallings graph $\Delta$ of the subgroup $H=\langle h_1,\ldots,h_n\rangle$ of $F$. By $1_\Delta$ we denote the base vertex of the graph $\Delta$. Let $\Gamma$ be a  linear graph labeled by a freely reduced word representing $g$, with initial vertex $1_\Gamma$ and terminal vertex $t_\Gamma$. Attach $\Gamma$ to $\Delta$ by identifying $1_\Gamma$ and $1_\Delta$ and perform Stallings foldings, marking the vertices of $\Gamma$ that get identified with $1_\Delta$. Let $\Delta_g$ be the obtained graph. Let $v$ be the farthest marked vertex along $\Gamma$. We claim that the group element $h_v\in H$ read as a label of the path in $\Gamma$ from $1_\Gamma$ to $v$ is the closest to $g$ element in $H$. Indeed, let $h\in H$ be such that $g=hg'$. Then the element $g'$ is readable in $\Delta_g$ as a label of a path from $1_\Delta$ to $t_\Gamma$. By construction, $g'$ cannot be shorter than the length of the path from $v$ to $t_\Gamma$ in $\Gamma$.
\end{proof}

Note that it is important for the above argument that the free group is given by its free generators. It is similarly important that in the below theorem the Coxeter or the surface group involved in the statement is given by its standard presentation. The case of an arbitrary presentation remains open.




\begin{theorem}\label{th:cox_and_surface_closest} There is polynomial time algorithm that, given a standard presentation of a Coxeter group or a surface group $G$ satisfying the Reduction Hypothesis, and elements $g,h_1,\ldots, h_m\in G$, finds the element $h\in H=\langle h_1,\ldots,h_m\rangle$ closest to $g$.
\end{theorem}
\begin{proof}
Similarly to the free group case, we consider the graph that consists of a bouquet of loops reading $h_1,\ldots,h_m$, and the acyclic graph $\Gamma$ that accepts all Dehn reduced forms of $g$. We apply procedure described in~\cite{Schupp:2003} in the proof of Theorems~\ref{th:coxeter_graph} and~\ref{th:surface_graph} to obtain the folded version of this graph,~$\Delta$. Now we find the shortest reduced word labeling a path from a vertex of $\Gamma$ identified with $1_\Delta$ to the terminus of image of~$\Gamma$ in~$\Delta$. Since geodesic words are Dehn reduced, this word will represent the closest to $g$ element of $H$.
\end{proof}

\begin{corollary}
Let $G$ be a Coxeter group or a surface group given by its standard presentation satisfying the Reduction Hypothesis. Then the Closest Element Problem in  $G$ is in $\P$.
\end{corollary}
\begin{proof}
Follows immediately from Theorem~\ref{th:cox_and_surface_closest}.
\end{proof}

\begin{theorem}\label{th:nilp_closest}
Let $G$ be a a finitely generated virtually nilpotent group. Then the Closest Element Problem in $G$ is polynomial time decidable for an arbitrary subgroup of $G$.
\end{theorem}
\begin{proof}
Indeed, let $H$ be a given subgroup of $G$ and $g\in G$. By a theorem of Wolf~\cite{Wolf} the growth of $G$ is polynomial, i.e. the ball $B_N$ of radius $N$ in Cayley graph of $G$ centered at $1$ contains polynomially many (in terms of $N$) elements of the group $G$. Moreover, by~\cite[Proposition 3.1]{MNU1} there is a $\P$-time algorithm to list all group elements contained in $B_N$. Now, note that the distance from $g$ to $H$ cannot exceed $|g|$. Set $N=|g|$ and find the closest to $g$ element of $H$ by bruteforcing elements of the form $gb$, $b\in B_N$. Since the size of $B_N$ is polynomial and the membership problem in finitely generated virtually nilpotent groups is decidable in polynomial time by Lemma~\ref{le:nilpotent_gauss}, this can be done in polynomial time.
\end{proof}

\section{Distance between rational subsets}
\label{se:rational}

In this section we consider the following problem.

\medskip\noindent
{\bf Rational subsets distance problem:} 
Given two rational subsets $R$ and $S$ of $G$, find the distance 
between $R$ and $S$, i.e., find two elements $r\in R$ and $s\in s$ 
with a shortest distance apart.

\begin{theorem}\label{th:subgroup_distance} 
Subgroup distance problem in a free group is decidable in polynomial time.
\end{theorem}
\begin{proof}
Given subgroups $H,G$ of the free group $F$, we construct in subquadratic time~\cite{T} their folded Stallings' graphs $\Gamma_1,\Gamma_2$, respectively. Then we build the product graph $\Delta=\Gamma_1\times\Gamma_2$.
Consider the connected component $\Delta_0$ of the base vertex $1_\Delta=(1_{\Gamma_1},1_{\Gamma_2})$. We may assume that $\Delta_0$ is a tree (otherwise $H\cap G$ is nontrivial).
For each vertex $v$ of $\Delta_0$, let $f_v$ be the element of the free group read along the reduced path from $1_\Delta$ to $v$. Further, let $h_v$ and $g_v$ be the elements  that realize distance from $f_v$ to $H,G$, respectively, that is, $f_vh_v\in H$, $f_vg_v\in G$ and $|h_v|$, $|g_v|$ minimal. Note that for each $v$ elements $h_v,g_v$ can be found in polynomial time by Theorem~\ref{th:free_shortest_elt}. Bruteforcing all vertices $v$ of $\Delta_0$, find $v=v_0$ with with minimal $|h_v|+|g_v|$. We claim that the shortest distance between $H$ an $G$ is attained on the pair of elements $f_{v_0}h_{v_0}, f_{v_0}g_{v_0}$.

Indeed, let $g\in G$ and $h\in H$ be the pair with minimal $|g^{-1}h|$. Then $h=wh'$, $g=wg'$, where there is no free cancellation in the product $(g')^{-1}h'$. Then the element $w$ is readable as a label of a path in $\Delta$ from $1_\Delta$ to some vertex $v$. Since $|g^{-1}h|=|g'|+|h'|$, the claim follows.
\end{proof}

Note that the key feature of Stallings graphs in the above argument is that they are deterministic automata. If a non-deterministic automata describing rational subsets of a free group are given as an input, one can of course produce deterministic ones and apply a similar argument. However, producing a deterministic automaton out of a non-deterministic one may result in an exponential blow-up in size, which invalidates polynomial complexity estimate. While we are not aware of a polynomial algorithm in general case, we can, of course, accept deterministic automata as an input and solve the corresponding problem in polynomial time.

\begin{theorem} For rational subsets given by deterministic automata, the Rational Subsets Distance Problem in a free group is decidable in polynomial time.
\end{theorem}
\begin{proof}
The proof repeats that of the subgroup case (Theorem~\ref{th:subgroup_distance}) with minor adjustments.
\end{proof}

As an immediate corollary we, for example, obtain that the shortest distance between cosets in a free group is solvable in polynomial time.

\section{Shortest  element problem}
\label{se:shortest}


\begin{theorem}\label{th:free_shortest_elt}
The shortest  element problem in free groups of finite rank is in $\P$.
\end{theorem}
\begin{proof}
For a subgroup $H$ of a free group given by elements represented by words $h_1,\ldots,h_n$, we construct its Stallings graph $\Gamma$ with initial vertex $1_\Gamma$. Then finding the shortest nontrivial element in $H$ is equivalent to finding shortest loop at $1_\Gamma$, which is well known to be polynomial time.
\end{proof}


\begin{theorem}\label{th:cox_and_surface_shortest} There is polynomial time algorithm that, given a standard presentation of a Coxeter group or a surface group $G$ satisfying the Reduction Hypothesis, and elements $h_1,\ldots, h_m\in G$, finds the shortest element $h\in H=\langle h_1,\ldots,h_m\rangle$.
\end{theorem}
\begin{proof}
Similarly to the free group case, we consider the graph $\Delta(H)$ provided by Theorem~\ref{th:coxeter_graph}, and find the shortest loop with a reduced label at $1_\Gamma$.
\end{proof}

\begin{corollary}
Let $G$ be a Coxeter group or a surface group given by its standard presentation satisfying the Reduction Hypothesis. Then the Shortest Element Problem in  $G$ is in $\P$.
\end{corollary}
\begin{proof}
Follows immediately from Theorem~\ref{th:cox_and_surface_shortest}.
\end{proof}

\begin{theorem}\label{th:nilp_shortest_elt}
The shortest  element problem in a finitely generated virtually nilpotent group $G$ is in $\P$.
\end{theorem}
\begin{proof}The proof is based on the same observation as that of Theorem~\ref{th:nilp_closest}. Indeed, let $H$ be a subgroup of $G$ given by words $h_1,\ldots,h_n$ in generators of $G$. Let $N$ denote the total length of these words. By a theorem of Wolf~\cite{Wolf} the growth of $G$ is polynomial, i.e. the ball $B_N$ of radius $N$ in Cayley graph of $G$ contains polynomially many elements of group $G$. Moreover, by~\cite[Proposition 3.1]{MNU1} there is a $\P$-time algorithm to list all group elements contained in $B_N$. Note that $H$ contains nontrivial elements if and only if at least one of the words $h_1,\ldots, h_n$ is nontrivial, i.e. if and only if $B_N$ contains at least one nontrivial element of $H$. Now, since the membership problem in finitely generated virtually nilpotent groups is decidable in polynomial time by Lemma~\ref{le:nilpotent_gauss}, we can find the shortest element in $G$ in polynomial time by brute search.
\end{proof}
\begin{remark}
The argument above is based on the fact that finitely generated virtually nilpotent groups have polynomial growth. By Gromov's theorem~\cite{Gromov_pgrowth:1981}  the converse is also true, i.e., polynomial growth implies virtual nilpotence, so the argument does not apply to any wider classes of groups.
\end{remark}

\section{Subgroup geodesic problem in a free group}

\label{se:geodesic-subgroup}
In this section we discuss a subgroup geodesic problem for
a finitely generated free group $F=F(X)$. It can be formulated as follows.
Given words $h_1,\ldots,h_m,h\in F$ express $h$ as an element in
$H=\gp{h_1,\ldots,h_m}$ in an optimal way, i.e., exprees $h$ as a product:
$$
h=h_{j_1}^{\varepsilon_1}\cdots h_{j_k}^{\varepsilon_k},
$$
with the least number of factors $k$. We prove that the problem can be solved
in polynomial time. 
Our main tool is an $X$-digraph with an additional labeling function.
Formally, we work with a tuple $(V,E,\mu,\nu)$, where:
\begin{itemize}
\item 
$(V,E)$ defines a directed graph.
\item 
$\mu:E\to X^\pm$.
\item 
$\nu:E\to F(h_1,\ldots,h_m)$, where $h_i$'s are formal letters that stand for 
the generators of $H$.
\end{itemize}
To achieve polynomial time complexity
we represent the values of $\nu$ using straight line programs 
(see \cite{Schleimer:2008}). For an edge $e\in E$ with $o(e)=v_1$, $t(v_2)=v_2$
$\mu(e)=x$ and $\nu(x)=u$ we often use the following notation: 
$$
e = v_1 \stackrel{x,u}{\to} v_2.
$$
The edge
$v_2 \stackrel{x^{-1},u^{-1}}{\to} v_1$ is called the \emph{inverse} of
$e=v_1 \stackrel{x,u}{\to} v_2$ and is denoted by $e^{-1}$.
A \emph{path} $p$ in a graph $\Gamma$ is a sequence of consecutive edges
$e_1,\ldots,e_k$; it has labels $\mu(p)=\mu(e_1)\ldots \mu(e_k)$
and $\nu(p)=\nu(e_1)\ldots\nu(e_k)$.
A \emph{circuit} in $\Gamma$ at $v\in V$ is a path $p$ such that $o(p)=t(p)=v$.

The initial step in the algorithm is to construct the graph 
$\Gamma_0=\Bouquet(h_1,\ldots,h_m)$.
For every word $h_i = y_1\ldots y_n \in F$ (where $y_j\in X^{\pm 1}$)
define a cyclic graph $\Gamma_{h_i} = (V,E)$, where:
$$V = \{0,\ldots,n-1\}$$
and
\begin{align*}
E =& \{i-1\stackrel{y_i,\varepsilon}{\to}i \mid 1\le i< n\} 
\cup \{n-1\stackrel{y_n,h_i}{\to}0\}\\
\cup & \{i\stackrel{y_i^{-1},\varepsilon}{\to}i-1 \mid 1\le i< n\} 
\cup \{0\stackrel{y_n^{-1},h_i^{-1}}{\to}n-1\}\\
\cup& \{i\stackrel{\varepsilon,\varepsilon}{\to}i\mid 0\le i<n\}.
\end{align*}
The graph $\Bouquet(h_1,\ldots,h_m)$ 
is obtained by merging the graphs $\Gamma_{h_1},\ldots,\Gamma_{h_m}$ 
at the vertex $0$. The vertex $0$ is designated as the \emph{root} of $\Gamma_0$.

\begin{figure}[htbp]
\centerline{ \includegraphics[height=4cm]{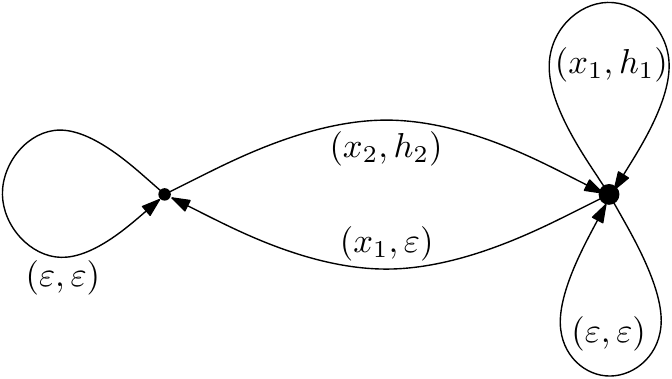} }
\caption{\label{fi:bouquet} $\Bouquet(x_1,x_1x_2)$. Inverse edges omitted.}
\end{figure}

Let $\gamma:F(h_1,\ldots,h_m) \rightarrow H$ be an epimorphism 
given by $\gamma(h_i) = h_i$.
The next lemma follows from the definition of $\Gamma_0$.

\begin{lemma}\label{le:basic_property}
Let $p$ be a circuit at $0$ in $\Gamma_0$. 
Then $\mu(p) = \gamma(\nu(p))$.
\qed
\end{lemma}

We say that a path $p$ in $\Gamma$ is {\em reduced} if $p$
does not involve a segment $e e^{-1}$. 
Reduction of a path in $\Gamma$ is a process of removing all 
segments $e e^{-1}$. It is not difficult to see
that the result of path-reduction is unique and does not depend 
on a particular choices of removals. 



We say that a pair of consecutive
edges $e_1 = v_1 \rightarrow v_2$ and $e_2 = v_2 \rightarrow v_3$
in $\Gamma$ is a {\em potential bypass} if:
\begin{itemize}
    \item
$|\mu(e_1)\mu(e_2)|\le 1$ and;
    \item
$\Gamma$ does not contain an edge $e = v_1 \rightarrow v_3$ such that
$\mu(e) = \mu(e_1)\mu(e_2)$.
\end{itemize}
Algorithm~\ref{al:completion} described below modifies the 
initial graph $\Gamma_0$ and produces a sequence of graphs:
$$
\Gamma_0 \subset \Gamma_1 \subset \ldots \subset \Gamma_s,
$$
where $\Gamma_{i+1}$ is obtained from $\Gamma_i$ by adding a single edge $e_i$ for 
some potential bypass $e_1,e_2$ at Step~\ref{step:edge}.
We point out that the output of Algorithm~\ref{al:completion}
may depend on choices it makes at Step \ref{step:short}.

\begin{algorithm}\caption{Completion: $\Gamma' = \Complete(\Gamma)$\label{al:complete}}
\label{al:completion}
\begin{algorithmic}[1]
\REQUIRE A bouquet graph $\Gamma = (V,E)$ as above.
\ENSURE A graph $\Gamma' = (V,E')$ with $E\subseteq E'$.
\WHILE{$\Gamma$ contains potential bypasses}
    \STATE \label{step:short}Find a potential bypass $e_1,e_2$ with the least $|\nu(e_1)\nu(e_2)|$.\label{al:least}
    \STATE \label{step:edge}Add a new edge $e = o(e_1) \rightarrow t(e_2)$
    \STATE \label{step:mu}Put $\mu(e) = \mu(e_1)\mu(e_2)$.
    \STATE \label{step:nu}Put $\nu(e) = \nu(e_1)\nu(e_2)$.
\ENDWHILE
\RETURN $\Gamma$.
\end{algorithmic}
\end{algorithm}

\begin{figure}[htbp]
\centerline{ \includegraphics[height=4cm]{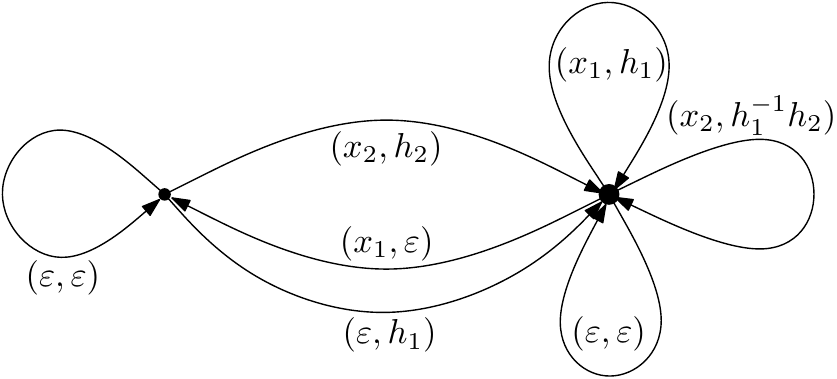} }
\caption{\label{fi:complete} Complete graph for $\{x_1,x_1x_2\}$.
Inverse edges omitted.}
\end{figure}

\begin{lemma}
Let $\Gamma' = \mathop{\mathrm{Complete}}(\Gamma)$.
Let $p$ be a circuit at the origin in $\Gamma'$. Then $\mu(p) = \gamma(\nu(p))$.
\end{lemma}

\begin{proof}
We prove the statement by induction for every graph $\Gamma_i$.
By Lemma \ref{le:basic_property}, $\mu(p) = \gamma(\nu(p))$
holds for every circuit in $\Gamma_0$.
Assume it holds for every circuit in $\Gamma_{i-1}$ and consider any 
circuit $p$ in $\Gamma_i$.
Let $e_i$ be the edge in $\Gamma_i$ added to $\Gamma_{i-1}$.
If $p$ does not involve $e_i$
then we may think that $p$ belongs to $\Gamma_{i-1}$
and hence $\mu(p) = \gamma(\nu(p))$. If $p$ involves the edge $e_i$ then
replacing every occurrence of $e_i$ by a subpath $e_i'e_i''$ we obtain a circuit $p'$
satisfying $\mu(p) = \mu(p')$ and $\nu(p) = \nu(p')$ and which does not involve $e_i$.
Therefore, $\mu(p) = \mu(p') = \gamma(\nu(p')) = \gamma(\nu(p))$.
\end{proof}

We say that a path $p$ in $\Gamma_0$ {\em defines a shortcut} if
\begin{itemize}
    \item[(S1)]
$|\mu(p)| \le 1$;
    \item[(S2)]
$\nu(p)$ is reduced;
    \item[(S3)]
$|\nu(p)|$ is the least possible among the paths with the same endpoints and the label $\mu(p)$.
\end{itemize}
Clearly, if $p$ defines a shortcut, then any its subpath $q$ 
with $|\mu(q)|\le 1$ defines shortcut.
We say that an edge $e$ in $\Gamma=\Complete(\Gamma_0)$ is a {\em shortcut} 
for a path $p$ in $\Gamma_0$ if $p$ defines a shortcut 
and $o(e) = o(p)$, $t(e) = t(p)$, $\mu(e) = \mu(p)$, and $\nu(e) = \nu(p)$.

The below lemma asserts that $\Gamma$ contains all shortcuts for paths 
in $\Gamma_0$ and, moreover, each edge in $\Gamma$ is a shortcut 
for some path in $\Gamma_0$.

\begin{lemma}
\label{le:complete}
Let $h_1,\ldots,h_m \in F\setminus \{\varepsilon\}$,
$\Gamma_0 = \Bouquet(h_1,\ldots,h_m)$, and $\Gamma = \Complete(\Gamma_0)$ .
The following holds.
\begin{itemize}
    \item[(E1)]
For every edge $e$ in $\Gamma$ there exists a path $p$ in $\Gamma_0$ such that
$o(p) = o(e)$, $t(p) = t(e)$, $\mu(p) = \mu(e)$, and $\nu(p) = \nu(e)$.
    \item[(E2)]
Furthermore, every edge
$e = v_1\rightarrow v_2\in\Gamma$ is a shortcut
for some path $p = p_e$ in $\Gamma_0$.
    \item[(E3)]
If a path $p$ from $v_1$ to $v_2$ in $\Gamma_0$ defines a shortcut, 
then $\Gamma$
contains an edge $e$ such that $\mu(e) = \mu(p)$ and $|\nu(e)| = |\nu(p)|$.
\end{itemize}
\end{lemma}

\begin{proof}
Clearly (E1) holds for every edge in $\Gamma_0$.
Assume that (E1) holds for every edge in $\Gamma_{i-1}$
and consider the edge $e\in\Gamma_i-\Gamma_{i-1}$, i.e.,
the potential bypass for edges $e',e''$ added on the $i$th step
to $\Gamma_{i-1}$. By the inductive assumption there exist paths $p',p''$
for $e',e''$ satisfying (E1). 
Then the path $p'p''$ is a required path for $e$. Thus, (E1) holds.

It is straightforward to check that for every edge $e$ in $\Gamma_0$
the path $p_e=e$ (consisting of a single edge $e$) is the required path.
Hence, (E2) holds for $\Gamma_0$.
Assume that (E2) holds for every edge in $\Gamma_{i-1}$
and consider the edge $e\in\Gamma_i-\Gamma_{i-1}$, i.e.,
the potential bypass for edges $e',e''$ added on the $i$th step
to $\Gamma_{i-1}$. By the inductive assumption there exist paths $p',p''$
for $e',e''$ satisfying (E2). If the path $p=p'p''$ does not define
a shortcut, then there is a shorter (reduced) path $q$ in $\Gamma_0$ 
with the origin $o(e)$, the terminus $t(p)$, and the label $\mu(q)=\mu(e)$
defining a shortcut.
Since $|\mu(q)|\le 1$, we can split $q$ into two nontrivial segments $q=q_1q_2$
each defining a shortcut.
Since Algorithm \ref{al:completion} on step \ref{step:short}
chooses a potential bypass with the least $\nu$-value, 
it follows that $\Gamma_{i-1}$ does not contain shortcuts for a path $q_1$ or $q_2$.
Assume without loss of generality that $\Gamma_{i-1}$ does not contain a shortcut 
for $q_1$. Proceeding as above for $q$, we split the path $q_1$ 
into two nontrivial segments $q_1=q_{11}q_{12}$
each defining a shortcut. 
Since instead of adding a shortcut for $q_1$ Algorithm \ref{al:completion} adds 
the edge $e_i$ it follows that $\Gamma_{i-1}$ does not contain a shortcut for 
$q_{11}$ or a shortcut for $q_{12}$, say $q_{11}$. And so on.
Proceeding in the same fashion, we eventually get a path $q'$ in $\Gamma_0$ of length 
$1$ (i.e., $q'$ is an edge) for which there is no shortcut.
But that is impossible because edge $\Gamma_0$ is a shortcut for itself.
The obtained contradiction proves that $e$ satisfies (E2).

One can prove (E3) the same way as (E2).
\end{proof}

\begin{algorithm}\caption{Geodesic: $u = \Geodesic(h_1,\ldots,h_k,w)$.\label{al:geodesic}}
\begin{algorithmic}[1]
\REQUIRE $\{h_1,\ldots,h_m\} \subset F_r\setminus\{\varepsilon\}$ 
and $w = x_{i_1}^{\varepsilon_1} \ldots x_{i_n}^{\varepsilon_n}\in F$.
\ENSURE A shortest factorization of $w$ in $h_i$'s.
\STATE $\Gamma_0 = \Bouquet(h_1,\ldots,h_m)$.
\STATE $\Gamma = \Complete(\Gamma_0)$.
\STATE $S = \{(0,\varepsilon)\}\subseteq V\times F(h_1,\ldots,h_m)$.
\FOR{$j=1$ to $n$}
    \STATE Put $S' = \emptyset$.
    \FORALL{$(v,u)\in S$}\label{step:for_S}
        \FORALL{$e = v\rightarrow v'$ in $\Gamma$ s.t. $\mu(e) = x_{i_j}^{\varepsilon_j}$}\label{step:forall_e}
            \STATE Put $S' = S' \cup\{(v',u\nu(e))\}$.
        \ENDFOR\label{step:endforall_e}
    \ENDFOR\label{step:endfor_S}
    \STATE Put $S = \emptyset$.
    \FORALL{$v\in V$}
        \IF{there exists $(v,u) \in S'$}
            \STATE Add a pair $(v,u)$ from $S'$ with the shortest $u$ to $S$.
        \ENDIF
    \ENDFOR
\ENDFOR
\STATE Find a pair $(0,u)\in S$.
\RETURN $u$.
\end{algorithmic}
\end{algorithm}

Now we turn to the Algorithm~\ref{al:geodesic} that computes geodesics in a given subgroup. The algorithm keeps a set $S$ of pairs $(v,u)$ where $v \in V$ and $u \in F(h_1,\ldots,h_m)$. We denote the initial set $S = \{(0,\varepsilon)\}$ by $S_0$
and the set $S$ obtained after $j$th iteration by $S_j$.

\begin{lemma}\label{le:description_S}
Let $w = x_{i_1}^{\varepsilon_1} \ldots x_{i_n}^{\varepsilon_n}\in F$.
Then for every $j\ge 1$ and every $(v,u) \in S_j$:
\begin{itemize}
    \item[(P1)]
The graph $\Gamma_0$ contains a path $p$ from $0$ to $v$ 
such that $\mu(p) = x_{i_1}^{\varepsilon_1} \cdots x_{i_j}^{\varepsilon_j}$,
$\nu(p)$ is reduced, and $\nu(p) = u$.
    \item[(P2)]
The word $u$ is a shortest possible $u=\nu(q)$ among all paths $q$ in $\Gamma_0$ 
from $0$ to $v$ with $\mu(q)=x_{i_1}^{\varepsilon_1} \cdots x_{i_j}^{\varepsilon_j}$.
    \item[(P3)]
Furthermore, if $\Gamma_0$ contains a path $p$ from $0$ to $v$ 
such that $\mu(P) = x_{i_1}^{\varepsilon_1} \ldots x_{i_j}^{\varepsilon_j}$,
then $S_j$ contains a pair $(v,u)$ for some $u \in F(h_1,\ldots,h_m)$.
\end{itemize}
\end{lemma}

\begin{proof}
Initially, $S = \{(0,\varepsilon)\}$.
To construct $S_1$ the algorithm constructs the set $S'$.
It easy to see that
\[S' = \{(t(e),\nu(e)) \mid e\in\Gamma\ \mbox{s.t.}\ o(e) = 0,\  \mu(e)=x_{i_1}^{\varepsilon_1})\}
\]
and $S_1 = S'$. It follows from Lemma \ref{le:complete} that (P1), (P2), and (P3) hold for $S_1$.

Assume that $S_{j-1}$ satisfies (P1), (P2), and (P3). To construct $S_j$
the algorithm constructs a set
\[
S' = \{(t(e),u\nu(e)) \mid (v,u) \in S_{j-1},\ e\in\Gamma \ \mbox{s.t.}\ o(e) = v,\ \mu(e)=x_{i_j}^{\varepsilon_j})\}.
\]
By induction assumption for every $(v,u) \in S_{j-1}$ there is a path 
$p_1$ in $\Gamma_0$
such that $o(p_1) = 0$, $t(p_1) = v$, $\mu(p_1) = x_{i_1}^{\varepsilon_1}\ldots x_{i_{j-1}}^{\varepsilon_{j-1}}$.
By Lemma \ref{le:complete} for every $e\in\Gamma$ there is a path $p_2$ in $\Gamma_0$ satisfying
$o(e) = o(p_2)$, $t(e) = t(p_2)$, $\mu(e) = x_{i_{j}}^{\varepsilon_{j}}$. Hence
$o(p_1p_2) = 0$, $t(p_1p_2) = t(e)$, $\mu(p_1p_2) = x_{i_1}^{\varepsilon_1}\ldots x_{i_j}^{\varepsilon_j}$.
Reducing $p_1p_2$ if necessary we obtain a path in $\Gamma_0$ that delivers (P1) for $(v,u)\in S'$. Since $S_j \subseteq S'$, the property (P1) holds for every pair in $S_j$.

Now we show (P2) and (P3). For a vertex $v$ of $\Gamma_0$ 
and for a given $1\le j\le n$, let $p$ be a path in $\Gamma_0$ from $0$ to $v$ 
with  $\mu(p) = x_{i_1}^{\varepsilon_1}\ldots x_{i_j}^{\varepsilon_j}$,  
$\nu(p)$ reduced and shortest possible.
Fixing a cancellation scheme for $\mu(p)$ 
we split the path $p$ into $p_1 \ldots p_j$ so that $\mu(p_\alpha) = x_{i_\alpha}^{\varepsilon_\alpha}$ for every $\alpha = 1,\ldots,j$.
Let $v' = t(p_1 \cdots p_{j-1})$.
Since $\nu(p)$ is reduced, it follows $(v',\nu(p_1\cdots p_{j-1}))\in S_{j-1}$.
By Lemma~\ref{le:complete}, $\Gamma$ contains an edge $e$ from $v'$ to $v$ such that
$\mu(e) = x_{i_{j}}^{\varepsilon_{j}}$. Hence $S'$ contains a pair $(v,\nu(p))$.
Since $\nu(p)$ is shortest possible for the vertex $v$ and the given $j$, it follows that $|u| = |\nu(p)|$.
Therefore, the properties (P2), (P3) hold for $S_j$.
\end{proof}

\begin{theorem}
Let $u = \Geodesic(h_1,\ldots,h_m,w)$. Then $\gamma(u) = w$ and $u$ is 
a shortest with such property, i.e., $u$ is geodesic for $w$. 
\end{theorem}

\begin{proof}
Let $w = x_{i_1}^{\varepsilon_1} \cdots x_{i_n}^{\varepsilon_n}$.
Assume that $w = h_{j_1}^{\delta_1}\cdots h_{j_k}^{\delta_k}$ and $k$ 
is the least possible.
The graph $\Gamma$ contains a circuit $p$ at $0$ such that 
$\mu(p) =  h_{j_1}^{\delta_1} \cdots  h_{j_m}^{\delta_k}$, 
where every $h_{j_\alpha}^{\delta_\alpha}$, $1\le\alpha\le k$, 
is traced along the corresponding loop.
By Lemma~\ref{le:description_S} the set $S_n$ contains a pair $(0,u)$
with $|u| = k$. Therefore the output is indeed a geodesic word for $w$.
\end{proof}

\begin{theorem}
Algorithms \ref{al:complete} and \ref{al:geodesic} terminate in polynomial time.
\end{theorem}

\begin{proof}
Algorithm \ref{al:complete} starts with the graph $\Gamma_0$.
To represent $\nu$-values on the edges it constructs a grammar 
consisting of terminal symbols $h_1^\pm,\ldots,h_m^\pm$ and non-terminal symbols
$\{S_e\mid e\in\Gamma'\}$. Adding an edge $e$ corresponding to a bypass $e_1,e_2$
adds a new non-terminal $S_e$ with production $S_{e_1}S_{e_2}$.
It follows from \cite{Schleimer:2008}, it requires polynomial time to compute
$|\nu(e_1)\nu(e_2)|$ for any pair of edges $e_1,e_2$.

Algorithm \ref{al:geodesic} processes $n$ letters of the word $w$.
It is easy to see that on every iteration $|S|\le |V|$ and $|S'|\le |V|^2$.
Choosing pairs $(v,u)$ with the least $|u|$ requires computing $|u|$,
which can be done in polynomial time. Hence, overall Algorithm \ref{al:geodesic}
has polynomial time complexity.
\end{proof}

\bibliography{main_bibliography}

\end{document}